\documentclass[12pt]{article}
\usepackage[utf8]{inputenc}
\usepackage[T1]{fontenc}

\usepackage{amsmath, amsthm}
\usepackage{amssymb}

\DeclareMathOperator{\parse}{parse}
\DeclareMathOperator{\maxFlxPalBase}{H}
\DeclareMathOperator{\maxFlxPal}{maxFlxPal}
\DeclareMathOperator{\revUniOccur}{ruo}
\DeclareMathOperator{\revUniOccurBase}{M}
\DeclareMathOperator{\reducedPrefix}{rdcPrf}
\DeclareMathOperator{\palclosure}{pc}
\DeclareMathOperator{\stdPalRep}{stdPalRep}
\DeclareMathOperator{\rw}{R}
\DeclareMathOperator{\lcp}{lcp}
\DeclareMathOperator{\lcs}{lcs}
\DeclareMathOperator{\lps}{lps}
\DeclareMathOperator{\lpps}{lpps}
\DeclareMathOperator{\lpp}{lpp}
\DeclareMathOperator{\lppp}{lppp}
\DeclareMathOperator{\Prefix}{Prf}
\DeclareMathOperator{\Suffix}{Suf}
\DeclareMathOperator{\StdExt}{StdExt}
\DeclareMathOperator{\MaxStdExt}{MaxStdExt}
\DeclareMathOperator{\rtrim}{rtrim}
\DeclareMathOperator{\ltrim}{ltrim}
\DeclareMathOperator{\trim}{trim}
\DeclareMathOperator{\occur}{occur}
\DeclareMathOperator{\reduced}{rdcWrd}
\DeclareMathOperator{\eliminated}{elmWrd}
\DeclareMathOperator{\FlxPal}{T}
\DeclareMathOperator{\Alphabet}{A}
\DeclareMathOperator{\Factor}{F}
\DeclareMathOperator{\MaxLengthWord}{MaxLenWord}
\DeclareMathOperator{\MinLengthWord}{MinLenWord}

\newcommand{\mlb}{2}

\newtheorem{theorem}{Theorem}[section]
\newtheorem{proposition}[theorem]{Proposition}
\newtheorem{corollary}[theorem]{Corollary}
\newtheorem{lemma}[theorem]{Lemma}

\newtheorem{definition}[theorem]{Definition}
\newtheorem{example}[theorem]{Example}
\newtheorem{remark}[theorem]{Remark}

\title{Construction Of A Rich Word Containing Given Two Factors}

\author{Josef Rukavicka\thanks{Department of Mathematics,
Faculty of Nuclear Sciences and Physical Engineering, CZECH TECHNICAL UNIVERSITY
IN PRAGUE
(josef.rukavicka@seznam.cz).}}

\begin{document}
\maketitle
\setcounter{footnote}{0}

\sloppy

\begin{abstract}
A finite word $w$ with $\vert w\vert=n$ contains at most $n+1$ distinct palindromic factors. If the bound  $n+1$ is attained, the word $w$ is called \emph{rich}. Let $\Factor(w)$ be the set of factors of the word $w$.
It is known that there are pairs of rich words that cannot be factors of a common rich word.
However it is an open question how to decide for a given pair of rich words $u,v$ if there is a rich word $w$ such that $\{u,v\}\subseteq \Factor(w)$. We present a response to this open question:\\
If $w_1, w_2,w$ are rich words, $m=\max{\{\vert w_1\vert,\vert w_2\vert\}}$, and $\{w_1,w_2\}\subseteq \Factor(w)$ then there exists also a rich word $\bar w$ such that $\{w_1,w_2\}\subseteq \Factor(\bar w)$ and $\vert \bar w\vert\leq  m2^{k(m)+2}$, where $k(m)=(q+1)m^2(4q^{10}m)^{\log_2{m}}$ and $q$ is the size of the alphabet. Hence it is enough to check all rich words of length equal or lower to $m2^{k(m)+2}$ in order to decide if there is a rich word containing factors $w_1,w_2$.
\end{abstract}

\section{Introduction}
In the last years there have appeared several articles dealing with rich words; see, for instance, \cite{BuLuGlZa2}, \cite{GlJuWiZa}, \cite{10.1007/978-3-319-66396-8_7}, \cite{Vesti2014}. Recall that a palindrome is a word that reads the same forwards and backwards, for example ``noon'' and ``level''. Rich words are those words that contain the maximal number of palindromic factors. It is known that a word of length $n$ can contain at most $n+1$ palindromic factors including the empty word. The notion of a rich word has been extended also to infinite words. An infinite word is called rich if its every finite factor is rich \cite{RUKAV_PALCOMPLXRICH_preprint}, \cite{10.1007/978-3-319-66396-8_7}. 
\\

Let $\lps(w)$ and $\lpp(w)$ denote the longest palindromic suffix and the longest palindromic prefix of a word $w$, respectively. The authors of \cite{BuLuGlZa2} showed the following property of rich words: 
\begin{proposition}
\label{tuob5ou6e}
If $r,t$ are two factors of a rich word $w$ such that $\lps(r)=\lps(t)$ and $\lpp(r)=\lpp(t)$, then $r=t$.
\end{proposition}
Two related open questions can be found: 
\begin{itemize}
\item
In \cite{Vesti2014}: Is the condition in Proposition \ref{tuob5ou6e} sufficient for joining two rich words $u$ and $v$ into factors of a same rich word?
\item
In \cite{10.1007/978-3-319-66396-8_7}: We do not know how to decide whether two rich words $u$ and $v$ are factors of a common rich word $w$.
\end{itemize}
In the current article we present a response to the question from \cite{10.1007/978-3-319-66396-8_7} in the following form: We prove that if $w_1, w_2,w$ are rich words, $m=\max{\{\vert w_1\vert,\vert w_2\vert\}}$, and $\{w_1, w_2\}\subseteq \Factor(w)$ then there exists a rich word $\bar w$ such that $\{w_1, w_2\}\subseteq \Factor(\bar w)$ and $\vert \bar w\vert\leq  m2^{k(m)+2}$, where $k(m)=(q+1)m^2(4q^{10}m)^{\log_2{m}}$ and $q$ is the size of the alphabet. Thus it is enough to check all rich words of length equal or lower to $m2^{k(m)+2}$ in order to decide if there is a rich word containing factors $w_1,w_2$.
\\

We describe the basic ideas of the proof.
If $w$ is a rich word, then let $a$ be a letter such that $\lps(wa)=a\lpps(w)a$, where $\lpps$ denotes the longest proper palindromic suffix. It is known and easy to show that $wa$ is a rich word \cite{Vesti2014}. Thus every rich word $w$ can be richly extended to a word $wa$. We will call $wa$ a \emph{standard extension} of $w$. If there is a letter $b$ such that $a\not =b$ and $wb$ is also a rich word, then we call the longest palindromic suffix of $wb$ a \emph{flexed palindrome}; the explication of the terminology is that $wb$ is not a \emph{standard extension} of $w$, hence $wb$ is ``flexed'' from the \emph{standard extension}. \\
We define a set $\Gamma$ of pairs of rich words $(w,r)$, where $r$ is a \emph{flexed palindrome} of $w$, the longest palindromic prefix of $w$ does not contain the factor $r$, and $\vert r\vert\geq \vert \bar r\vert$ for each \emph{flexed palindrome} $\bar r$ of $w$. If $(w,r)\in \Gamma$, $w_1$ is the prefix of $w$ with $\vert w_1\vert=\vert r\vert-1$ and $w_2$ is the suffix of $w$ with $\vert w_2\vert=\vert r\vert-1$ then we construct a rich word $\bar w$ possessing the following properties:
\begin{itemize}
\item
The word $w_1$ is a prefix of $\bar w$.
\item
The word $w_2$ is a suffix of $\bar w$.
\item
The number of occurrences of $r$ in $\bar w$ is strictly smaller than the number of occurrences of $r$ in $w$.
\item
The set of \emph{flexed palindromes} of $\bar w$ is a subset of the set of \emph{flexed palindromes} of $w$.
\end{itemize}
Iterative applying of this construction will allow us for a given rich word $w$ with a prefix $w_1$ and a suffix $w_2$ to construct a rich word $t$ containing factors $w_1,w_2$ and having no \emph{flexed palindrome} longer than $m$, where $m=\max\{\vert w_1\vert, \vert w_2\vert\}$.\\
Another important, but simple, observation is that if $w$ is a rich word with prefix $u$ such that the number of \emph{flexed palindromes} in $w$ is less than $k$ and $u$ has exactly one occurrence in $w$ then there is an upper bound for the length of $w$. We show this upper bound as a function of $k$ and consequently we derive an upper bound for the length of $t$. 

\section{Preliminaries}
Let $\Alphabet$ be a finite alphabet with $q=\vert\Alphabet\vert$. The elements of $\Alphabet$ will be called letters.\\
Let $\epsilon$ denote the empty word.\\
Let $\Alphabet^*$ be the set of all finite words over $\Alphabet$ including the empty word, let $\Alphabet^n\subset \Alphabet^*$ be the set of all words of length $n$, and let $\Alphabet^+=\Alphabet^*\setminus \{\epsilon\}$.\\  
Let $\rw\subset \Alphabet^*$ denote the set of all rich words and let $\rw^+=\rw\cap \Alphabet^+$.\\
Let $\Factor(w)\subset \Alphabet^*$ denote the set of all factors of $w\in \Alphabet^*$; we state explicitly that $\epsilon, w\in \Factor(w)$. \\
Let $\Factor(S)=\bigcup_{v\in S}\Factor(v)$, where $S\subseteq \Alphabet^*$.\\
Let $\Factor_p(w)\subseteq\Factor(w)$ be set of all palindromic factors of $w\in \Alphabet^*$.\\
Let $\Factor(w,r)=\{u\mid u\in \Factor(w)\mbox{ and }r\not \in \Factor(u)\}\subseteq \Factor(w)$, where $w,r\in \Alphabet^*$. The set $F(w,r)$ contains factors of $w$ avoiding the factor $r$.\\
Let $\Factor_p(w,r)=\Factor_p(w)\cap\Factor(w,r)$.\\
Let $\Prefix(w)$ and $\Suffix(w)$ be the set of all prefixes and all suffixes of $w\in A^*$ respectively; we define that $\{\epsilon, w\}\subseteq \Prefix(w)\cap\Suffix(w)$.\\
Let $w^R$ denote the reversal of $w\in A^*$; formally if $w=w_1w_2\dots w_k$ then $w^R=w_k\dots w_2w_1$, where $w_i\in \Alphabet$ and $i\in \{1,2,\dots,k\}$. In addition we define that $\epsilon^R=\epsilon$.\\
\\
Let $\lps(w)$ and $\lpp(w)$ denote the longest palindromic suffix and the longest palindromic prefix of $w\in \Alphabet^*$ respectively. We define that $\lps(\epsilon)=\lpp(\epsilon)=\epsilon$.\\ 
Let $\lpps(w)$ and $\lppp(w)$ denote the longest proper palindromic suffix and the longest proper palindromic prefix of $w\in \Alphabet^*$ respectively, where $\vert w\vert\geq 2$.\\ 
Let $\trim(w)=v$, where $v,w\in \Alphabet^*$, $x,y\in \Alphabet$, $w=xvy$, and $\vert w\vert\geq 2$.\\
Let $\rtrim(w)=v$, where $v,w\in \Alphabet^*$, $y\in \Alphabet$, $w=vy$, and $\vert w\vert\geq 1$.\\
Let $\ltrim(w)=v$ , where $v,w\in \Alphabet^*$, $x\in \Alphabet$, $w=xv$, and $\vert w\vert\geq 1$.

\begin{example}
\begin{itemize}
\item $\Alphabet=\{1,2,3,4,5\}$.
\item $w=124135$.
\item $\trim(w)=2413$.
\item $\ltrim(w)=24135$.
\item $\rtrim(w)=12413$.
\end{itemize}
\end{example}
\noindent
Let $\palclosure(w)$ be the palindromic closure of $w\in \Alphabet^*$; formally $\palclosure(w)=uvu^R$, where $w=uv$ and $v=\lps(w)$.\\
\\
Let $\MinLengthWord(U)$ and $\MaxLengthWord(U)$ be the shortest and the longest word from the set $U$ respectively, where either $U\subseteq \Prefix(w)$ or $U\subseteq \Suffix(w)$ for some $w\in A^*$. If $U=\emptyset$ then we define $\MinLengthWord(U)=\epsilon$ and $\MaxLengthWord(U)=\epsilon$.\\
\\
Let $\lcp(w_1,w_2)$ be the longest common prefix of words $w_1,w_2\in \Alphabet^*$; formally $\lcp(w_1,w_2)=\MaxLengthWord(\Prefix(w_1)\cap\Prefix(w_2))$.\\
Let $\lcs(w_1,w_2)$ be the longest common suffix of words $w_1,w_2\in \Alphabet^*$; formally $\lcs(w_1,w_2)=\MaxLengthWord(\Suffix(w_1)\cap\Suffix(w_2))$.\\
Let $\occur(u,v)$ be the number of occurrences of $v$ in $u$, where $u,v\in \Alphabet^*$ and $\vert v\vert>0$; formally $\occur(u,v)=\vert \{w\mid w\in \Suffix(u)\mbox{ and }v\in \Prefix(w)\}\vert$.\\

Recall the notion of a \emph{complete return} \cite{GlJuWiZa}: Given a word $w$ and factors $r,u\in \Factor(w)$, we call the factor $r$ a \emph{complete return} to $u$ in $w$ if $r$ contains exactly two occurrences of $u$, one as a prefix and one as a suffix. \\

We list some known properties of rich words that we use in our article. All of them can be found, for instance, in  \cite{GlJuWiZa}.
\begin{proposition}
\label{yy65se85bj5d}
If $w,u\in R^+$ and $u\in \Factor_p(w)$ then all complete returns to $u$ in $w$ are palindromes.
\end{proposition}
\begin{proposition}
If $w \in \rw$ and $p\in \Factor(w)$ then $p, p^R\in \rw$.
\end{proposition}
\begin{proposition}
\label{oij5498fr654td222gh}
A word $w$ is rich if and only if every prefix $p\in \Prefix(w)$ has a unioccurrent palindromic suffix.
\end{proposition}

\section{Standard Extensions and Flexed Palindromes}
We start with a formal definition of a \emph{standard extension} and a \emph{flexed palindrome} introduced at the beginning of the article.
\begin{definition}
Let $j\geq 0$ be a nonnegative integer, $w\in \rw$, and $\vert w\vert\geq \mlb$. We define $\StdExt(w,j)$ as follows:
\begin{itemize}
\item
$\StdExt(w,0)=w$.
\item
$\StdExt(w,1)=wa$ such that $\lps(wa)=a\lpps(w)a$ and $a\in \Alphabet$.
\item
$\StdExt(w,j)=\StdExt(\StdExt(w,j-1),1)$, where $j>1$.
\end{itemize}
Let $\StdExt(w)=\{\StdExt(w,j)\mid j\geq 0\}$. If $p\in \StdExt(w)$ then we call $p$ a \emph{standard extension} of $w$. 
\\
Let 
$\FlxPal(w) = \{\lps(ub)\mid ub\in \Prefix(w)\mbox{ and }b\in \Alphabet\mbox{ and }ub\not=\StdExt(u,1)\}$. If $r\in \FlxPal(w)$ then we call $r$ a \emph{flexed palindrome} of $w$.
\end{definition}

For a given rich word $w\in \rw$ having a \emph{flexed palindrome} $r$ we define a \emph{standard palindromic replacement} of $r$ to be the longest palindromic suffix of a \emph{standard extension} of a prefix $p$ of $w$ such that $\lps(px)=r$, where $px$ is a prefix of $w$ and $x\in \Alphabet$. The idea is that we can ``replace'' $r$ with the \emph{standard palindromic replacement}.
\begin{definition}
Let $\stdPalRep(w,r)=\lps(\StdExt(h,1))$, where $w,r\in \rw$,  $r\in \FlxPal(w)$, $hx\in \Prefix(w)$, $x\in \Alphabet$, and $\lps(hx)=r$.\\ We call $\stdPalRep(w,r)$ a \emph{standard palindromic replacement} of $r$ in $w$.
\end{definition}

\begin{example}
\begin{itemize}
\item
$\Alphabet=\{0,1\}$.
\item
$w=110101100110011$.
\item
$001100\in \FlxPal(w)$.
\item
$\lps(1101011001100)=001100$.
\item
$\StdExt(110101100110,1)=1101011001101$.
\item
$\stdPalRep(w,001100)=\lps(1101011001101)=1011001101$.
\end{itemize}
\end{example}
We show that the length of a flexed palindrome $r$ is less than the length of the \emph{standard palindromic replacement} $\stdPalRep(w,r)$.
\begin{lemma}
\label{tlmq54flp8j8tt}
If $ux,uy\in \rw$, $x,y\in \Alphabet$, $x\not= y$, and $ux=\StdExt(u,1)$ then $\vert\lps(ux)\vert>\vert \lps(uy)\vert$.
\end{lemma}
\begin{proof}
Let $yty=\lps(uy)$. From the definition of a \emph{standard extension} we have $\lps(ux)=xvx$, where $v=\lpps(u)$ and hence $t\in \Suffix(v)$. Since $y\not =x$ we have also $yt\in \Suffix(v)$. The lemma follows. 
\end{proof}
An obvious corollary is that a \emph{flexed palindrome} of $w$ is not a prefix of $w$. 
\begin{corollary}
\label{tuipp568tr6}
If $w,r\in \rw$ and $r\in \FlxPal(w)$ then $r\not \in \Prefix(w)$.
\end{corollary}
\noindent
In \cite{Vesti2014} the \emph{standard extension} has been used to prove that each rich word $w$ can be extended ``richly''; this means that there is $a\in A$ such that $wa$ is rich. \begin{lemma}
\label{uyyye52fggh58o}
If $w\in \rw$ and $\vert w\vert\geq 2$ then $\StdExt(w)\subset R$.
\end{lemma}
\begin{proof}
Obviously it is enough to prove that $\StdExt(w,1)\in \rw$, since for every $t\in \StdExt(w)\setminus\{w\}$ there is a rich word $\bar t$ such that $t=\StdExt(\bar t,1)$.\\ 
Let $xpx=\lps(\StdExt(w,1))$, where $x\in \Alphabet$. Proposition \ref{oij5498fr654td222gh} implies that we need to prove that $xpx$ is unioccurrent in $\StdExt(w,1)$. Realize that $p$ is unioccurrent in $w$, hence $xpx$ is unioccurrent in $\StdExt(w,1)$.
\end{proof}
\noindent
To simplify the proofs of the paper we introduce a function $\MaxStdExt(u,v)$ to be the longest prefix $z$ of $u$ such that $z$ is also a \emph{standard extension} of $v$:
\begin{definition} Let
$\MaxStdExt(u,v)=\MaxLengthWord( \{\StdExt(v)\cap \Prefix(u)\})$. We call $\MaxStdExt(u,v)$ a \emph{maximal standard extension} of $v$ in $u$.
\end{definition}

The next lemma shows that if a rich word contains factors $ypx$ and $ypy$, where $p$ is a palindrome, $p$ is not a prefix of $w$, $x,y$ are distinct letters, and $ypx$ ``occurs'' before $ypy$ in $w$ then $ypy$ is a \emph{flexed palindrome}.
\begin{lemma}
\label{lkjtoi654fro89}
If $w,v,p\in\rw$, $v\in \Prefix(w)$, $p\not \in \Prefix(w)$, $x,y\in \Alphabet$, $x\not=y$, $ypx\in \Suffix(v)$, $ypy\not \in \Factor(v)$, and $ypy\in \Factor(w)$ then $ypy\in \FlxPal(w)$.
\end{lemma}
\begin{proof}
Let $\bar v$ be such that $\bar vy\in \Prefix(w)$, $ypy\in \Suffix(\bar vy)$, and $\occur(\bar vy, ypy)=1$. Let $u=\lps(\bar v)$. Because $p\not \in \Prefix(w)$ it follows that $u=\lpps(\bar v)=\lps(\bar v)$ and thus there is $z\in \Alphabet$ such that $zu\in \Suffix(\bar v)$. Obviously $v\in \Prefix(\bar v)$ and hence $\occur(\bar v,p)>1$. Proposition \ref{yy65se85bj5d} implies that $\occur(u,p)>1$, since the complete return to $p$ which is a suffix of $\bar v$ must a suffix of $u$. It follows that $yp\in \Suffix(u)$ and Lemma \ref{tlmq54flp8j8tt} implies that $ypy\in \FlxPal(w)$.
\end{proof}

\section{Removing flexed points}
We define formally the set $\Gamma$ mentioned in the introduction. 
An element $(w,r)$ of the set $\Gamma$ represents a rich word $w$ for which we are able to construct a new rich word $\bar w$ such that $\bar w$ does not contain the \emph{flexed palindrome} $r$, but $\bar w$ have certain common prefixes and suffixes with $w$. We define that $r$ is one of the longest \emph{flexed palindromes} of $w$ and that $r$ is not a factor of the longest palindromic prefix of $w$. In addition we require that $\vert r\vert>2$ so that the \emph{standard extension} of $\rtrim(r)$ would be defined.
\begin{definition}
\label{ppjsgammao58l}
Let $\Gamma$ be a set defined as follows: $(w,r)\in \Gamma$ if
\begin{enumerate}
\item
$w,r\in \rw$ and
\item
$\vert r\vert>2$.
\item
$r\in \FlxPal(w)$ and
\item \label{lkjowprot585t}
$r\not \in \Factor(\lpp(w))$ and
\item \label{lkjowprot58522}
$\vert r\vert \geq \vert \bar r\vert$ for each $\bar r\in \FlxPal(w)$.
\end{enumerate}
\end{definition}

Given $(w,r)\in \Gamma$, we need to express $w$ as a concatenation of its factors having some special properties. For this reason we define a function $\parse(w,r)$:
\begin{definition}
\label{yuuumparse5866z}
If $(w,r)\in \Gamma$ then let $\parse(w,r)=(v,z,t)$, where
\begin{itemize}
\item
$v,z,t\in \rw$ and
\item
$vzt=w$ and
\item
$r\in \Suffix(v)$ and
\item
$\occur(w,r)=\occur(v,r)$ and
\item
$vz=\MaxStdExt(vzt,v)$.
\end{itemize}
\end{definition}
\begin{remark}
The prefix $v$ is the shortest prefix of $w$ that contains all occurrences of $r$. The prefix $vz$ is the \emph{maximal standard extension} of $v$ in $w$, and $t$ is such that $vzt=w$. It is easy to see that $v,z,t$ exist and are uniquely determined for $(w,r)\in \Gamma$.
\end{remark}
For an element $(w,r)\in \Gamma$ we define a function $\reducedPrefix(w,r)$ (a \emph{reduced prefix}), which is a prefix of the palindromic closure of some prefix of $w$. In Theorem \ref{knmgh85emn7k} we show that the concatenation of $\reducedPrefix(w,r)$ and $t$ is a rich word having a strictly smaller number of occurrences of $r$ than in $w$, where $(v,z,t)=\parse(w,r)$.
\begin{definition}
\label{p58prt114g}
If $w,r\in \Gamma$ and $(v,z,t)=\parse(w,r)$ then let
$\reducedPrefix(w,r)$ be defined as follows:\\
It follows from Property \ref{lkjowprot585t} of Definition \ref{ppjsgammao58l} that there is $h\in \Prefix(w)$ such that $w=hz^R\lps(v)zt$. Note that $\lps(v)\not=v$ since $r\in \FlxPal(w)$ and thus $r\not \in \Prefix(w)$, see Corollary \ref{tuipp568tr6}. It is clear that $r\in \Prefix(\lps(v))\cap\Suffix(\lps(v))$. This implies that $ hz^Rr\in \Prefix(w)$. We distinguish two cases:
\begin{itemize}
\item
$r\in \Factor(hz^R\rtrim(r))$:\\
Let $g$ be the complete return to $r$ such that $g\in \Suffix(hz^Rr)$. Clearly $rz\in \Prefix(g)$ and $z^Rr\in \Suffix(g)$, since $r\not \in\Factor(\ltrim(r)z)$; recall $r\in \Suffix(v)$ and $\occur(v,r)=\occur(vzt,r)$. Let $\bar g$ be such that $\bar gg=hz^Rr$.\\ We define $\reducedPrefix(w,r)=\bar grz$. Note that $\reducedPrefix(w,r)\in \Prefix(w)$.
\item
$r\not \in \Factor(hz^R\rtrim(r))$:\\
Let $\bar u=\stdPalRep(hz^Rr,r)$.
Clearly $\lps(hz^Rr)=r$ and $\bar u\not =r$. Because $z^R\rtrim(r)\in \Suffix(hz^R\rtrim(r))$, then obviously $U\not=\emptyset$ and $r\not \in \Factor(U)$, where $U=\{u\mid u\in \Prefix(\palclosure(hz^R\rtrim(r)))\mbox{ and }\ltrim(r)z\in \Suffix(u)\}$. We define $\reducedPrefix(w,r)=\MinLengthWord(U)$. Note that $r\not \in \Factor(\reducedPrefix(w,r))$.
\end{itemize}
We call $\reducedPrefix(w,r)$ a \emph{reduced prefix} of $w$ by $r$.
\end{definition}
\begin{remark}
Note in Definition \ref{p58prt114g} in the second case where $r\not \in \Factor(hz^R\rtrim(r))$ that it may happen that $\reducedPrefix(w,r)$ is not a prefix of $w$. However it is a prefix of a palindromic closure of $hz^R\rtrim(r)$, hence the number of \emph{flexed palindromes} remains the same; formally $\vert \FlxPal(hz^R\rtrim(r)))\vert=\vert \FlxPal(\reducedPrefix(w,r))\vert$. Realize that $\palclosure(t)\in \StdExt(t)$ for each $t\in \rw$ and $\vert t\vert\geq 2$.
\end{remark}
\noindent
To clarify the definition of the \emph{reduced prefix} $\reducedPrefix(w,r)$ we present below two examples representing those two cases in the definition.
\begin{example}
\begin{itemize}
\item
$\Alphabet=\{1,2,3,4,5,6,7,8,9\}$.
\item
$w=123999322399932442399932255223993$.
\item
$r=999$.
\item
$v=1239993223999324423999$.
\item
$z=322$.
\item
$t=55223993$.
\item
$\lps(v)=999324423999$.
\item
$h=1239993$.
\item
$w=hz^R\lps(v)zt$.
\item
$g=9993223999\in \Suffix(hz^Rr)=\Suffix(1239993223999)$.
\item
$\bar g=123$.
\item
$\reducedPrefix(w,r)=123999322$.
\end{itemize}
\end{example}
\begin{example}
\begin{itemize}
\item
$\Alphabet=\{1,2,3,4,5,6,7,8,9\}$.
\item
$w=123999599932239949$.
\item
$r=999$.
\item
$v=1239995999$.
\item
$z=32$.
\item
$t=239949$.
\item
$\lps(v)=9995999$.
\item
$h=1$.
\item
$w=hz^R\lps(v)zt$.
\item
$\StdExt(hz^R\rtrim(r), 1)=\StdExt(12399,1)=123993$.
\item
$\bar u=\stdPalRep(123999,999)=3993$.
\item
$\palclosure(12399)=12399321$.
\item
$U=\{1239932\}$.
\item
$\reducedPrefix(w,r)=1239932$.
\end{itemize}
\end{example}
We know that the \emph{reduced prefix} $\reducedPrefix(w,r)$ may not be a prefix of $w$, however we show that the longest common prefix of $\reducedPrefix(w,r)$ and $w$ is longer than $\vert r\vert-1$.
\begin{lemma}
\label{ru539e5txx2}
If $(w,r)\in \Gamma$ and $u=\reducedPrefix(w,r)$ then $\vert\lcp(u,w)\vert\geq \vert r\vert -1$.
\end{lemma}
\begin{proof}
In Definition \ref{p58prt114g} in the first case where $r\not \in \Factor(hz^R\rtrim(r))$ and $u\in \Prefix(w)$, it is clear that $r\in \Factor(u)$ and thus $\vert u\vert \geq \vert r\vert$. Hence we need to verify only the second case, where $r\not \in \Factor(hz^R\rtrim(r))$. Either $hz^R\rtrim(r)\in \Prefix(u)$ or $u\in \Prefix(hz^R\rtrim(r))$. Since $\ltrim(r)z^R\in \Suffix(u)$ the lemma follows.
\end{proof}
\noindent
Using the \emph{reduced prefix} we can now define the word $\reduced(w,r)$ (a \emph{reduced word}):
\begin{definition}
Let $\reduced(w,r)=\reducedPrefix(w,r)t$, where $(v,z,t)=\parse(w,r)$ and $(w,r)\in\Gamma$. We call $\reduced(w,r)$ a \emph{reduced word} of $w$ by $r$.
\end{definition}
We show that the longest common suffix of the \emph{reduced word} $\reduced(w,r)$ and $w$ is longer than $\vert r\vert -1$.
\begin{lemma}
\label{ru539e5q441d}
If $(w,r)\in \Gamma$ then $\lcs(\reduced(w,r),w)\vert\geq \vert r\vert -1$.
\end{lemma}
\begin{proof}
Given $(w,r)\in \Gamma$ and $(v,z,t)=\parse(w,r)$.
From Definition \ref{p58prt114g} of the \emph{reduced prefix}, it is obvious that $\ltrim(r)z\in \Suffix(\reducedPrefix(w,r))$ and consequently $\ltrim(r)zt\in \Suffix(\reduced(w,r))$. Recall Definition \ref{yuuumparse5866z} of the function $\parse(w,r)$. Since $w=vzt$ and $r\in \Suffix(v)$ it follows that $\ltrim(r)zt\in \Suffix(w)$. This implies that $\ltrim(r)zt$ is a common suffix of $w$ and $\reduced(w,r)$. Because $\vert\ltrim(r)\vert=\vert r\vert -1$ the lemma follows.
\end{proof}
As already mentioned the \emph{reduced prefix} $\reducedPrefix(w,r)$ is not necessarily a prefix of $w$. In such a case $\reducedPrefix(w,r)$ contains palindromic factors that are not factors of the longest common prefix $\lcp(w,\reducedPrefix(w,r))$. We show that none of these palindromes is a factor of $w$. This will be important when proving richness of the word $\reduced(w,r)$. 
\begin{proposition}
\label{eek54w56pk}
If $(w,r)\in \Gamma$, $u=\reducedPrefix(w,r)$, $\bar u=\stdPalRep(w,r)$, and $g=\lcp(w,u)$, then $\Factor_p(u,\bar u)\subseteq \Factor_p(g)$ and $\bar u\not \in \Factor_p(w)$. 
\end{proposition}
\begin{proof}
From the properties of the palindromic closure it is easy to see that $\Factor_p(\palclosure(f),\lps(f))\subseteq\Factor_p(f)$ for each $f\in \rw$. It means that every palindromic factor of $f$ that is not a factor of $\palclosure(f)$ contains the factor $\lps(f)$. It follows that $\Factor_p(u,\bar u)\subseteq \Factor_p(g)$.\\
We show that $\occur(w,\bar u)=0$. Let $\bar u=xtx$ and $r=ypy$, where $x,y\in \Alphabet$. Obviously $x\not=y$, $py\in \Prefix(t)$, and $yp\in \Suffix(p)$. Thus $xty\in \Factor(w)$. Lemma \ref{lkjtoi654fro89} implies that $\bar u\in \Factor(w)$ if and only if $\bar u\in \FlxPal(w)$. In addition Lemma \ref{tlmq54flp8j8tt} implies that $\vert\bar u\vert>\vert r\vert$. This is a contradiction to Property \ref{lkjowprot58522} of Definition \ref{ppjsgammao58l}. Hence $\bar u\not \in \Factor_p(w)$. This completes the proof. 
\end{proof}

The main theorem of the paper states the the \emph{reduced word} $\reduced(w,r)$ is rich, where $(w,r)\in \Gamma$. In addition the theorem asserts that the set of \emph{flexed palindromes} of $\reduced(w,r)$ is a subset of the set of \emph{flexed palindromes} of the word $w$, the number of occurrences of $r$ is strictly smaller in $\reduced(w,r)$ than in $w$, and the longest common prefix and suffix of $\reduced(w,r)$ and $w$ are longer than $\vert r\vert-1$. 
\begin{theorem}
\label{knmgh85emn7k}
If $(w,r)\in \Gamma$ then 
\begin{itemize}
\item
$\reduced(w,r)\in \rw$ and
\item
$\FlxPal(\reduced(w,r))\subseteq\FlxPal(w)$ and
\item
$\occur(\reduced(w,r),r)<\occur(w,r)$ and
\item
$\vert\lcp(\reduced(w,r),w)\vert\geq \vert r\vert-1$ and
\item
$\vert\lcs(\reduced(w,r),w)\vert\geq \vert r\vert-1$.
\end{itemize}
\end{theorem}
\begin{proof}
Recall that $\reduced(w,r)=ut$, where $(v,z,t)=\parse(w,r)$  and $u=\reducedPrefix(w,r)$.
Suppose that $up\in \rw$, $\FlxPal(up)\subseteq \FlxPal(vzp)$, where $p\in \Prefix(\rtrim(t))$. In addition suppose that if $\vert p\vert\geq 1$ then $\lps(up)=\lps(vzp)$ and $r\not \in \Factor(\lps(vzp))$. The assumptions obviously hold for $p=\epsilon$. \\
Let $x\in \Alphabet$ be such that $px\in \Prefix(t)$. We show that the assumptions hold also for $upx$.\\
Proposition \ref{oij5498fr654td222gh} implies that if $f\in \rw$ and $y\in \Alphabet$ then $fy\in \rw$ if and only if $fy$ has a unioccurrent palindromic suffix. Using this property we prove the theorem. We distinguish two cases: 
\begin{itemize}
\item
If $\lps(vzpx)\in \FlxPal(w)$ then Property \ref{lkjowprot58522} of Definition \ref{ppjsgammao58l} implies that $\lps(vzpx)\in\Suffix(\ltrim(r)zpx)$. From Definition \ref{p58prt114g} we know that $\ltrim(r)z\in \Suffix(u)$. This implies that $\lps(vzpx)\in \Suffix(upx)$ and $r\not \in \Factor(\lps(vzpx))$.\\ Proposition \ref{eek54w56pk} implies that $\lps(vzpx)$ is unioccurrent in $upx$. In consequence $\lps(upx)=\lps(vzpx)$ and $upx\in \rw$. Because $\lps(vzpx)\in \FlxPal(w)$ and $\FlxPal(up)\subseteq\FlxPal(vzp)$ we conclude that $\FlxPal(upx)\subseteq\FlxPal(vzpx)$. We do not need to prove that $\lps(upx)\in \FlxPal(upx)$, although it would not be difficult.
\item
If $\lps(vzpx)\not \in \FlxPal(w)$, then $\vert p\vert\geq 1$, because $vz=\MaxStdExt(vzt,v)$. Realize that $\lps(vzy)\in \FlxPal(w)$, where $y\in \Alphabet$ and $vzy\in \Prefix(vzt)$.\\ Hence according to our assumptions we have $\lps(up)=\lps(vzp)$. Obviously $\lps(vzpx)=x\lpps(vzp)x$ and $r\not \in \lpps(vzp)$. \\ Suppose that $r\in \Factor(\lps(vzpx))$. Then $r\in \Prefix(\lps(vzpx))\cap \Suffix(\lps(vzpx))$. This is a contradiction since $\occur(v,r)=\occur(w,r)$, see Definition \ref{yuuumparse5866z}. This implies that $r\not \in\lps(vzpx)$. It follows that $\vert\lps(vzpx)\vert<\vert rzpx\vert$ and that $\lps(vzpx)\in \Suffix(upx)$. In consequence $\lps(upx)=x\lpps(up)x$. Thus $upx$ is a \emph{standard extension} of $up$. We conclude that $\lps(upx)=\lps(vzpx)$, $\lps(upx)\not \in \FlxPal(upx)$, $\FlxPal(upx)\subseteq\FlxPal(vzpx)$, and $upx\in \rw$.
\end{itemize}
So we have $upx\in \rw$ and $\FlxPal(upx)\subseteq\FlxPal(vzpx)$ for each $px\in \Prefix(t)$. The fact that $\occur(ut,r)<\occur(w,r)$ follows simply from the construction of $u=\reducedPrefix(w,r)$, see Definition \ref{p58prt114g}. \\
Lemma \ref{ru539e5txx2} and Lemma \ref{ru539e5q441d} imply that $\vert\lcp(\reduced(w,r),w)\vert\geq \vert r\vert-1$ and
$\vert\lcs(\reduced(w,r),w)\vert\geq \vert r\vert-1$. This completes the proof. 
\end{proof}
\noindent
Two more examples illuminate the construction of $\reduced(w,r)$.
\begin{example}
\begin{itemize}
\item
$\Alphabet=\{1,2,3,4,5,6,7,8\}$.
\item
$w=vzt=12145656547745656545656547874$.
\item
$r=656$.
\item
$v=12145656547745656545656$.
\item
$z=547$.
\item
$t=874$.
\item
$\lps(v)=656545656$.
\item
$u=\reducedPrefix(w,r)=12145656547$.
\item
$\reduced(w,r)=ut=12145656547874$.
\end{itemize}
\end{example}
\begin{example}
\begin{itemize}
$\Alphabet=\{1,2,3,4,5,6,7,8\}$.
\item
$w=vzt=12145656547874$.
\item
$r=656$.
\item
$v=12145656$.
\item
$z=54$.
\item
$t=7874$.
\item
$\lps(v)=656$.
\item
$u=\reducedPrefix(w,r)=12145654$.
\item
$\reduced(w,r)=ut=121456547874$.
\end{itemize}
\end{example}

If a rich word $w$ has a factor $u$, then the palindromic closure of $w$ is rich and contains the factor $u^R$. Hence for us when constructing a rich word containing given factors, it does not matter if $w$ contains $u$ or $u^R$. We introduce the notion of a \emph{reverse-unioccurrent} factor.
\begin{definition}
If $\vert\{u,u^R\}\cap \Factor(w)\vert=1$ then we say that a word $u$ is \emph{reverse-unioccurrent} in $w$, where $w,u\in \rw$.
\end{definition}
We introduce a function $\revUniOccur(w,u,v)$ (a \emph{reverse unioccurrence} of $u,v$ in $w$) which returns a factor of $w$ such that $u,v$ are \emph{reverse unioccurrent}. In addition we require that $u$ or $u^R$ is a prefix and $v$ or $v^R$ is a suffix of $\revUniOccur(w,u,v)$.
\begin{definition}
\label{uuh52r236e}
If $w_1,w_2,w\in \rw$, $w_1\in \Prefix(w)$ and $w_2\in \Suffix(w)$, then let $\revUniOccurBase(w,w_1,w_2)\subset \Factor(w)$ such that $t\in \revUniOccurBase(w,w_1,w_2)$ if:
\begin{itemize}
\item
$t\in \Factor(w)$ and
\item
$w_1,w_2$ are \emph{reverse-unioccurrent} in $t$ and
\item
$\{w_1,w_1^R\}\cap\Prefix(t)\not =\emptyset$ and
\item
$\{w_2,w_2^R\}\cap\Suffix(t)\not =\emptyset$.
\end{itemize}
Let the set $\revUniOccurBase(w,w_1,w_2)$ be ordered and let $\revUniOccur(w,w_1,w_2)$ be the first element of $\revUniOccurBase(w,w_1,w_2)$.
\end{definition}
\begin{remark}
It is not difficult to see that the function $\revUniOccur(r,w_1,w_2)$ is well defined and the set $\revUniOccurBase(w,w_1,w_2)$ is nonempty.
\end{remark}
We define \emph{maximal flexed palindrome} of a rich word $w$, which is a \emph{flexed palindrome} $r$ of $w$, such $(w,r)\in \Gamma$ and $\vert r\vert>n$, where $n$ is a positive integer.
\begin{definition}
Let $\maxFlxPalBase(w,n)=\{r\mid (w,r)\in \Gamma\mbox{ and }\vert r\vert>n \}$, let the set $\maxFlxPalBase(w,n)$ be ordered and let $\maxFlxPal(w,n)$ be the first element of $\maxFlxPalBase(w,n)$. If $\maxFlxPalBase(w,n)=\emptyset$ then we define $\maxFlxPal(w,n)=\epsilon$. We call $\maxFlxPal(w,n)$ a \emph{maximal flexed palindrome} of $w$.
\end{definition}
\begin{remark}
The name ``\emph{maximal flexed palindrome}'' comes from the properties of $\Gamma$. Recall that for a pair $(w,r)$ to be in the set $\Gamma$, it is necessary that $r$ is one of the longest \emph{flexed palindromes} of $w$.
\end{remark}

Next we define the function $\eliminated(w,w_1,w_2)$ (\emph{eliminated word}) that constructs a rich word from $w$ by ``eliminating all'' \emph{flexed palindromes} longer than $m=\max\{\vert w_1\vert, \vert w_2\vert\}$ and keeping the prefix $w_1$ and the suffix $w_2$ of $w$.
\begin{definition}
If $w,w_1,w_2\in \rw$, $m=\max\{\vert w_1\vert, \vert w_2\vert\}$, $w_1\in \Prefix(w)$, and $w_2\in \Suffix(w)$, then let $\eliminated(w,w_1,w_2)$ be the result of the following procedure:
\begin{verbatim}
01 INPUT: w,m,w_1,w_2;
02 res: = ruo(w,w_1,w_2);
03 r := maxFlxPal(res,m);
04 WHILE r is nonempty word 
05 DO
06  res := rdcWrd(res,r);
07  res := ruo(res,w_1,w_2);
08  r := maxFlxPal(res,m);
09 END-DO;
10 RETURN res;
\end{verbatim}
\end{definition}
The call of the function $\revUniOccur$ on the lines $02$ and $07$ guarantees that $w_1,w_2$ are \emph{reverse-unioccurrent} in the word $res$ and that $\{w_1,w_1^R\}\cap \Prefix(res)\not =\emptyset$ and $\{w_2, w_2^R\}\cap \Suffix(res)\not =\emptyset$. Realize that it is not guaranteed that $w_1,w_2$ are \emph{reverse-unioccurrent} in $\reduced(res,r)$, even if $w_1,w_2$ are \emph{reverse-unioccurrent} in $res$.\\
Clearly, the facts that $\bar t$ is \emph{reverse unioccurrent} in a rich word $t$ and $\bar t\in \Prefix(t)$ imply that $\lppp(t)\in \Prefix(\bar t)$. Thus if $r$ is a \emph{flexed palindrome} of $t$ longer than the prefix $\bar t$, then $r$ is not a factor of $\lppp(t)$ and hence $r$ satisfies Property \ref{lkjowprot585t} of Definition \ref{ppjsgammao58l}. In consequence the word $\eliminated(w,w_1,w_2)$ contains no \emph{flexed palindrome} longer than $m$.\\
The call of the function $\reduced(res,r)$ on the line $06$ makes obviously sense, since if $\maxFlxPal(w,m)\not =\epsilon$ then $(w,\maxFlxPal(w,m))\in \Gamma$. \\
In addition, because $\vert r\vert>\max\{\vert w_1, w_2\}$, Theorem \ref{knmgh85emn7k} asserts that $\{w_1,w_1^R\}\cap \Prefix(\reduced(res,r))\not =\emptyset$ and $\{w_2, w_2^R\}\cap \Prefix(\reduced(res,r))\not =\emptyset$; consequently $\{w_1,w_1^R\}\cap \Prefix(res)\not =\emptyset$ and $\{w_2,w_2^R\}\cap\Suffix(res)\not =\emptyset$ on the line $06$. \\
Moreover Theorem \ref{knmgh85emn7k} implies that the procedure finishes after a finite number of iterations, because $\occur(\reduced(w,r),r)<\occur(w,r)$ and $\FlxPal(\reduced(w,r))\subseteq \FlxPal(w)$. The number of iterations is bounded by the number $\sum_{r\in \FlxPal(w)}\occur(w,r)$. Note that several occurrences of $r$ may be ``eliminated'' in one iteration. Hence we proved the following lemma:
\begin{lemma}
\label{nb23cc25bn}
If $(w,r)\in \Gamma$, $w_1\in \Prefix(w)$, $w_2\in \Prefix(w)$, $m=\max\{\vert w_1\vert, \vert w_2\vert\}$, and $t=\eliminated(w,w_1,w_2)$ then
\begin{itemize}
\item
$t\in \rw$ and 
\item
$\{w_1,w_1^R\}\cap\Prefix(t)\not =\emptyset$ and 
\item
$\{w_2,w_2^R\}\cap\Suffix(t)\not=\emptyset$ and 
\item
for each $r\in \FlxPal(t)$ we have $\vert r\vert\leq m$.
\end{itemize}
\end{lemma}


\section{Words with limited number of flexed points}
What is the maximal length of a word $u$ such that $w$ is \emph{reverse-unioccurrent} in $u$, $w$ is a prefix of $u$, and $u$ has a given maximal number of \emph{flexed palindromes}? The proposition below answers this question.
\begin{proposition}
\label{oiun3218eknn}
If $u,w\in \rw^+$, $w\in \Prefix(u)$, $\vert\FlxPal(u)\setminus\FlxPal(w)\vert\leq k$, $\vert w\vert \leq m$, and $w$ is \emph{reverse-unioccurrent} in $u$ then $\vert u\vert\leq m2^{k+1}$.
\end{proposition}
\begin{proof}
Let $\bar u=\StdExt(u,1)$; then obviously $\vert\palclosure(\bar u)\vert<2\vert \bar u\vert$, $\palclosure(\bar u)\in \StdExt(u)$, and $w$ is not \emph{reverse-unioccurrent} in $\palclosure(\bar u)$, since $w^R\in \Suffix(\palclosure(\bar u))$.\\
It follows that if $v_1,v_2\in \Prefix(\bar u)$ such that $v_1$ is \emph{reverse unioccurrent} in $\bar u$, $v_1\in \Prefix(v_2)$, $\vert\FlxPal(v_2)\setminus\FlxPal(v_1)\vert=1$, and $\lps(v_2)\in \FlxPal(v_2)$ then $\vert \ltrim(v_2)\vert< 2\vert v_1\vert$, since $\ltrim(v_2)\in \StdExt(v_1)$. This implies that $\vert v_2\vert\leq 2\vert v_1\vert$. The proposition follows.
\end{proof}
\begin{remark}
The proof asserts that if $v_1,v_2$ are two prefixes of a word $u$ such that the longest palindromic suffix of $v_2$ is the only \emph{flexed palindrome} in $v_2$ which is not a factor of $v_1$, then $v_2$ is at most twice longer than $v_1$ on condition that $v_1^R$ is not a factor of $\ltrim(v_2)$. Less formally it means that the length of a word can grow at most twice before next \emph{flexed palindrome} appears. Note that for $k=1$ we have $\vert u\vert\leq 2m$, which makes sense, since the palindromic closure of a nonpalindromic  word $w$ is at most twice longer than $w$ and $w$ is not \emph{reverse-unioccurrent} in $\palclosure(w)$; realize that $w^R\in \Suffix(\palclosure(w))$. 
\end{remark}

\noindent
In \cite{RUKAV_PALCOMPLXRICH_preprint} the author showed an upper bound for the number of palindromic factors of given length in a rich word:
\begin{proposition}[\cite{RUKAV_PALCOMPLXRICH_preprint},Corollary 2.23]
\label{truksip356ju}
If $w\in \rw$ and $n>0$ then $$\vert F_p(w) \cap \Alphabet^n\vert\leq (q+1)n(4q^{10}n)^{\log_2{n}}\mbox{.}$$
\end{proposition}

\noindent
Proposition \ref{truksip356ju} implies an upper bound for the number of \emph{flexed palindromes}:
\begin{lemma}
\label{trir58r5pk}
If $w\in \rw$, $n>0$, and $\FlxPal(w)\cap \Alphabet^j=\emptyset$ for each $j>n$ then $$\vert \FlxPal(w)\vert \leq (q+1)n^2(4q^{10}n)^{\log_2{n}}\mbox{.}$$
\end{lemma}
\begin{proof}
Just realize that $\sum_{j=1}^n (q+1)j(4q^{10}j)^{\log_2{j}}\leq (q+1)n^2(4q^{10}n)^{\log_2{n}}$.
\end{proof}

\noindent
From Lemma \ref{nb23cc25bn}, Lemma $\ref{trir58r5pk}$ and Proposition \ref{oiun3218eknn} we obtain the result of the article:
\begin{corollary}
\label{ottlwq596plk}
If $w,w_1,w_2$ are rich words, $w_1,w_2\in \Factor(w)$, $m=\max{\{\vert w_1\vert,\vert w_2\vert\}}$ then there exists also a rich word $\bar w$ such that $w_1,w_2\in\Factor(\bar w)$ and $\vert \bar w\vert\leq m2^{k(m)+2}$, where $k(m)=(q+1)m^2(4q^{10}m)^{\log_2{m}}$.
\end{corollary}
\begin{proof}
Let $t\in \Factor(\palclosure(w))$ such that $w_1\in \Prefix(t)$ and $w_2\in \Suffix(t)$. Obviously such $t$ exists. 
Consider the word $g=\eliminated(t,w_1,w_2)$. Let $k(m)=(q+1)m^2(4q^{10}m)^{\log_2{m}}$. Lemma $\ref{trir58r5pk}$ and Proposition \ref{oiun3218eknn} imply that $\vert g\vert\geq m2^{k(m)+1}$. Lemma \ref{nb23cc25bn} implies that $g\in \rw$,  $\{w_1,w_1^R\}\cap\Factor(g)\not=\emptyset$, and $\{w_2,w_2^R\}\cap\Factor(g)\not=\emptyset$. Let $\bar w=\palclosure(g)$. It follows that $w_1,w_2\in \Factor(\bar w)$.
Because $\vert\palclosure(g)\vert\leq 2\vert g\vert$, the corollary follows.	
\end{proof}

\section*{Acknowledgments}
The author wishes to thank to Štěpán Starosta for his useful comments. The author acknowledges support by the Czech Science
Foundation grant GA\v CR 13-03538S and by the Grant Agency of the Czech Technical University in Prague, grant No. SGS14/205/OHK4/3T/14.

\bibliographystyle{siam}
\IfFileExists{biblio.bib}{\bibliography{biblio}}{\bibliography{../!bibliography/biblio}}

\begin{thebibliography}{1}

\bibitem{BuLuGlZa2}
{\sc M.~Bucci, A.~{De Luca}, A.~Glen, and L.~Q. Zamboni}, {\em A new
  characteristic property of rich words}, Theor. Comput. Sci., 410 (2009),
  pp.~2860--2863.

\bibitem{GlJuWiZa}
{\sc A.~Glen, J.~Justin, S.~Widmer, and L.~Q. Zamboni}, {\em Palindromic
  richness}, Eur. J. Combin., 30 (2009), pp.~510--531.

\bibitem{10.1007/978-3-319-66396-8_7}
{\sc E.~Pelantov{\'a} and {\v{S}}.~Starosta}, {\em On words with the zero
  palindromic defect}, in Combinatorics on Words, S.~Brlek, F.~Dolce,
  C.~Reutenauer, and {\'E}.~Vandomme, eds., Cham, 2017, Springer International
  Publishing, pp.~59--71.

\bibitem{RUKAV_PALCOMPLXRICH_preprint}
{\sc J.~Rukavicka}, {\em An upper bound for palindromic and factor complexity
  of rich words}, preprint available at https://arxiv.org/abs/1810.03573,
  submitted for publication,  (2018).

\bibitem{Vesti2014}
{\sc J.~Vesti}, {\em Extensions of rich words}, Theor. Comput. Sci., 548
  (2014), pp.~14--24.

\end{thebibliography}

\end{document}